\documentclass[a4paper,12pt]{amsart}
\usepackage[utf8]{inputenc}
\usepackage[T1]{fontenc}
\usepackage{fullpage}
\usepackage[english]{babel}
\usepackage{mathrsfs}
\usepackage{pgfplots}
\usepackage{amsmath, amsthm, amssymb}
\usepackage[retainorgcmds]{IEEEtrantools} 
\title{Intrinsic Hardy-Orlicz Spaces of conformal mappings}
\author{Pekka Koskela and Sita Benedict}
\newtheorem{theorem}{Theorem}[section]
\newtheorem{lemma}[theorem]{Lemma}

\newtheorem{corollary}[theorem]{Corollary}

\newtheorem*{theorem*}{Theorem}
\newtheorem*{proposition*}{Proposition}
\newtheorem*{corollary*}{Corollary}
\newtheorem*{ghthm}{Gehring-Hayman Theorem}
\newtheorem*{remark*}{Remark}
\numberwithin{equation}{section}
 
\newcommand{\D}{\mathbb{D}} 
\newcommand{\Sn}{\partial\D} 

\begin{document}

\begin{abstract}
We define a new type of Hardy-Orlicz spaces of conformal mappings on the unit disk where in place of the value $|f(x)|$ we consider the intrinsic path distance between $f(x)$ and $f(0)$ in the image domain. We show that if the Orlicz function is doubling then these two spaces are actually the same, and we give an example when the intrinsic Hardy-Orlicz space is strictly smaller. 
\end{abstract}

\footnotetext{
\endgraf The authors were partially supported by the Academy of Finland grants
131477 and 263850.}

\maketitle
\section{Introduction}
Let $\psi: [0, \infty] \rightarrow [0, \infty]$ be a differentiable and strictly increasing function such that $\psi(0) = 0$; that is, a \textit{growth function}. A conformal map $f: \mathbb{D} \rightarrow \mathbb{C}$ belongs to the Hardy-Orlicz space $H^\psi$ if there exists $\delta > 0$ such that
\begin{eqnarray*}
\sup_{0<r<1}\int_{\Sn}\psi(\delta|f(r\omega)|)d\sigma < \infty,
\end{eqnarray*}
and to $H^\infty$ if
\begin{eqnarray*}
\sup_{0<r<1} M(r,f) < \infty,
\end{eqnarray*}
where
\begin{eqnarray*}
M(r,f) = \sup_{\omega \in \Sn} \{|f(r\omega)| \}
\end{eqnarray*} 
is the maximum modulus function on $0 < r< 1$. Some results on Hardy-Orlicz spaces can be found in \cite{hpsiqc}, \cite{hardyorlicz}, \cite{deeb2}, \cite{deeb1}, \cite{khalil}, and \cite{pcarleson}, however the exact definition of the spaces varies and the theory is not limited to conformal mappings on $\D$. In the case $\psi(t) = t^p, 0 < p < \infty$, see \cite{durenhp}. 

Since a conformal map induces a change of metric in $\D$, the intrinsic path distance  $d_I(f(x), f(0))$ in $f(\Omega)$ is in many occasions more natural than $|f(x)|$, see for example \cite{conformalmetrics}. We abbreviate ${|f(x)|_I = d_I(f(x), f(0))}$ and say that $f$ belongs to the intrinsic Hardy-Orlicz space $H_I^\psi$ if there exists $\delta > 0$ such that
\begin{eqnarray*}\label{hpsiintdef}
\sup_{0<r<1}\int_{\Sn}\psi(\delta|f(r\omega)|_I)d\sigma < \infty,
\end{eqnarray*}
and to $H_I^\infty$ if $|f(x)|_I$ is uniformally bounded on $\D$. Here $d\sigma$ denotes integration with respect to the length measure on $\Sn$.

If $f(\mathbb{D})$ is a bounded domain then $f\in H^\infty$ by definition. It is easy to see that there are conformal mappings that belong to $H^\infty$ but do not belong to $H^\infty_I$; one only has to consider a conformal mapping of the unit disk onto a bounded domain that spirals inward ad infinitum. For general $\psi$ it is not obvious whether there are conformal mappings belonging to $H^\psi$ but not to $H^\psi_I$. In this paper we show somewhat surprisingly that if $\psi$ is doubling then there are no such mappings.

\begin{theorem}\label{mainthm}
Let $f$ be a conformal mapping of $\D$. If $\psi$ is a doubling growth function then $f \in H^\psi$ if and only if $f \in H_I^\psi$. 
\end{theorem}

Thus, our theorem gives a new characterization of the conformal mappings belonging to the classical $H^p$ spaces for all $0 < p < \infty$. Using the Gehring-Hayman inequality we also obtain the following. 
\begin{corollary}\label{mainthmcor}
Let $f$ be a conformal mapping of $\D$ and $\psi$ a doubling growth function. Then\begin{eqnarray*}
f \in H^\psi \;\;\; \textnormal{if and only if}\;\;\; \int_{\Sn} \psi \left(\int_0^1 |f'(r\omega)| dr \right) d\sigma < \infty.
\end{eqnarray*}
\end{corollary}
We will give examples to show that if the $\psi$ is not doubling then the result may fail, see Section 5.

The authors would like to dedicate this paper to both Kari Astala and Michel Zinsmeister, whose ideas in \cite{hpqc} and \cite{zins} respectively have contributed to several of the techniques used in this paper. 

\section{Basic Definitions}
Denote by $\mathbb{D} = \{x\in \mathbb{C} : |x| < 1 \}$ the open unit disk in the complex plane, by $\partial\D$ its boundary the unit circle, by $B(x,r)$ the open disk centered at $x \in \mathbb{C}$ with radius $r > 0$ and by $\partial B(x, r)$ its boundary. An analytic function $f: \Omega \rightarrow \mathbb{C}$ on a domain $\Omega \subset \mathbb{C}$ is called conformal if it is also injective. In this paper the conformal mappings we consider will all have as domain $\D$.

A curve in a domain $\Omega \subset \mathbb{C}$ is a continuous mapping $\gamma: I \rightarrow \Omega$  of an interval $I\subseteq \mathbb{R}$. The image set $\gamma(I)$ is also denoted by $\gamma$.  The euclidean length of the curve is denoted by length$(\gamma) \in [0, \infty]$. If length$(\gamma) < \infty$ then $\gamma$ is \textit{rectifiable}. A curve is called \textit{locally rectifiable} if length$(\gamma|_{[a,b]})<\infty$ for every closed sub-interval $[a,b] \subset I$. We allow that the endpoints of a curve in $\D$ may belong to $\Sn$.

Given a simply connected domain $\Omega$ strictly contained in $\mathbb{C}$, the Riemann Mapping Theorem asserts the existence of a conformal function $f$ mapping $\mathbb{D}$ onto $\Omega$. Recall that if $\gamma$ is a curve in $\mathbb{D}$ then $f\circ \gamma$ is a curve in $\Omega$, and when integrating with respect to arc length we have
\begin{eqnarray*}
\text{length}(f\circ\gamma) = \int_\gamma |f'(z)|\: |dz|.
\end{eqnarray*}
Given $u, v \in \Omega$ the intrinsic path distance between $u$ and $v$ in $\Omega$ is
\begin{eqnarray*}
d_I(u,v) := \inf \text{length}(\gamma),
\end{eqnarray*}
where the infimum is taken over all curves $\gamma$ in  $\Omega$ with endpoints $u$ and $v$. This defines a metric on $\Omega$. If $f$ maps $\D$ conformally onto $\Omega$ and $x = f^{-1}(u)$ and $y = f^{-1}(v)$ then clearly
\begin{eqnarray*}
d_I(u, v) = \inf \text{length}(f \circ \gamma),
\end{eqnarray*}
where the infimum is taken over all curves in $\D$ with endpoints $x$ and $y$. We abbreviate then $d_I(f(x), f(0))$ to $|f(x)|_I$ for each $x\in\mathbb{D}$, and we denote usual metric space notions with respect to $d_I$ with the addition of the subscript $_I$. For instance, the diameter of a set $E\subset \Omega$ in the metric $d_I$ will be denoted by $\textnormal{diam}_I(E)$.    

It is well known that if $f$ is a conformal mapping of $\D$ then the radial limit
\begin{eqnarray*}
\lim_{r\rightarrow1}f(r\omega) = f(\omega)
\end{eqnarray*}
exists and length($f([0, \omega)) < \infty$ for almost every $\omega \in \Sn$, see for example \cite{pommerenke}. If $f(\omega)$ exists we define for each $x\in\D$ the intrinsic path distance
\begin{eqnarray*}
d_I(f(x), f(\omega)) = \inf \text{length}(f \circ \gamma),
\end{eqnarray*}
where the infimum is taken over all curves $\gamma$ in $\D$ with endpoints $x$ and $\omega$, and we abbreviate like before 
\begin{eqnarray*}
|f(\omega)|_I := d_I(f(0), f(\omega)).
\end{eqnarray*}

We will mostly be maneuvering through $\D$ via 'Whitney-type disks,' their corresponding 'shadows' on $\Sn$ as well as the related Stolz cones, which we define here. Given $x \in \D$ let the Whitney disk centered at $x$ be defined as $B_x = B(x, \frac{1-|x|}{2})$, and let $S_x = \{\frac{z}{|z|}: z\in B_x\}$ denote its shadow on $\Sn$. For each $\omega \in \Sn$ let
\begin{eqnarray*}
\Gamma(\omega) = \bigcup\{B_{t\omega} : 0 \leq t < 1\}
\end{eqnarray*}
be the Stolz cone at $\omega$. We associate with each conformal mapping $f: \D \rightarrow \mathbb{C}$ its non-tangential maximal function with respect to the euclidean metric as
\begin{eqnarray*}
f^*(\omega) = \sup_{x\in\Gamma(\omega)} |f(x)|, \omega \in \Sn,
\end{eqnarray*}
and its non-tangential maximal function with respect to $d_I$ as
\begin{eqnarray*}
f_I^*(\omega) = \sup_{x\in\Gamma(\omega)} |f(x)|_I, \omega \in \Sn.
\end{eqnarray*}. 

\section{Koebe, Modulus estimates, and Gehring-Hayman}
We begin this section by stating two corollaries to the well-known Koebe distortion theorem, see \cite[Corollaries 1.4 and 1.5]{pommerenke}.
\begin{lemma}\label{koebecorr2}
There exists a universal constant $C$ such that if $f$ is a conformal mapping of $\D$ then
\begin{eqnarray}\label{distafequiv}
\frac{1}{C}d(f(x), \partial f(\D)) \leq |f'(x)|(1-|x|) \leq Cd(f(x), \partial f(\D)) 
\end{eqnarray}
for all $x\in \D$.
\end{lemma}
The notation $d_h(x,y)$ in the next statement denotes the hyperbolic metric on $\mathbb{D}$.
\begin{lemma}\label{koebecorr}
If $f:\mathbb{D}\rightarrow \mathbb{C}$ is conformal then
\begin{eqnarray*}
e^{-6d_h(x,y)} \leq \frac{|f'(x)|}{|f'(y)|} \leq e^{6d_h(x,y)}
\end{eqnarray*}
for any $x,y\in \mathbb{D}.$
\end{lemma}
This lemma tells us that $|f'|$ is roughly constant on Whitney disks $B_x = B(x, \frac{1-|x|}{2})$, $x \in \D$. Indeed if $x, y \in B_z$ for some $z\in\mathbb{D}$ then $d_h(x,y) \leq 2$, and so by setting $C = e^{12}$ we have
\begin{align}\label{harnack}
\frac1C\leq  \frac{|f'(x)|}{|f'(y)|}\leq C\;\;\text{for all}\;x,y\in B_z, z\in \D.
\end{align}
It follows easily that there is an absolute constant $C$ such that
\begin{align}\label{internaldiamupperbound}
\textnormal{diam}(f(B_x)) \leq \textnormal{diam}_I(f(B_x)) \leq Cd(f(x), \partial f(\D))
\end{align}
for all $x\in \D$.

We collect some basic facts about modulus of curve families, needed for the lemmas that follow. If $\Gamma$ is a family of locally rectifiable curves in a domain $\Omega \subseteq \mathbb{C}$, a Borel function $\rho: \Omega \rightarrow [0,\infty]$ is called \textit{admissible} if 
\begin{eqnarray*}
\int_\gamma \rho\; ds \geq 1
\end{eqnarray*}
for every curve $\gamma \in \Gamma$.
The modulus Mod($\Gamma) \in [0, \infty]$ of the curve family $\Gamma$ is then defined as
\begin{eqnarray*}
\inf \int_\Omega \rho^2\; dm,
\end{eqnarray*}
where the infimum is taken over all admissible $\rho$. Here $dm$ denotes integration with respect to Lebesgue measure in the plane. 

An important property of the modulus of curve families is that it is a conformal invariant. That is, when $f: \Omega \rightarrow \Omega'$ is a conformal map between domains and $\Gamma$ is a family of curves in $\Omega$, then 
\begin{eqnarray*}
\text{Mod}(\Gamma) = \text{Mod}(f\Gamma),
\end{eqnarray*}
where $f\Gamma = \{f \circ \gamma : \gamma \in \Gamma \}$. 

It is possible to compute the modulus or at least arrive at a useful estimate for certain families of curves. If $E$ is a Borel set in $\Sn$, $0<r<1$, and $\Gamma$ is the family of radial segments joining $\partial B(0,r)$ to $E$ then 
\begin{eqnarray}
\text{Mod}(\Gamma) = \frac{\sigma(E)}{(\log 1/r)}.
\end{eqnarray}
More generally, the upper bound
\begin{eqnarray}
\text{Mod}(\Gamma) \leq \frac{\sigma(\partial\D)}{\log(R/r)}
\end{eqnarray}
is valid whenever each $\gamma \in \Gamma$ joins $\partial B(x, r)$ to $\partial B(x, R)$, $0 < r < R$. 
These and other basic properties of the modulus can be found in \cite{vaisala}. 

The next lemma is simply a special case of \cite[Lemma 3.2]{conformalmetrics}, and so we omit 
its proof here. 

\begin{lemma}\label{confdensmodest}
There exists a universal constant $C$ with the following property. Let $\Omega$ be a simply connected domain in $\mathbb{C}$ equipped with the intrinsic metric $d_I$, $E$ a non-empty subset of $\Omega$ and suppose $L \geq \delta > 0$. Assume that $\textnormal{diam}_I(E) \leq \delta$ and that $\Gamma$ is a family of curves in $\Omega$ so that each curve $\gamma \in \Gamma$ has one endpoint in $E$ and $\textnormal{length}(\gamma) \geq L$. Then
\begin{align*}
\textnormal{Mod}(\Gamma) \leq \frac{C}{\log (1 + L/\delta)}.
\end{align*}
\end{lemma}
The next two lemmas are used in Section 4 for our main result. Their proofs are similar, using basic modulus of curve families techniques. 
\begin{lemma}\label{moduluscons}
Let $f$ be a conformal mapping of $\D$ and $M > 1$. There exists an absolute constant $C$ such that
\begin{align*}
\sigma(\{\omega \in S_x : d_I(f(\omega), f(x)) > M d(f(x), \partial f(\D)) \}) \leq C\sigma(S_x)(\log M)^{-1}
\end{align*}
for every $x\in \D$.
\end{lemma}

\begin{proof}
Fix $x\in \D$ and set $d = d(f(x), \partial f(\D))$ and $E = \{\omega \in S_x:  d_I(f(\omega), f(x)) > M d\}$. First suppose $|x| < 1/4$ and let $\Gamma_E$ be the set of radial segments that have one endpoint in $E$ and the other in $B_x \cap \partial B(0,1/4)$. By (\ref{internaldiamupperbound}) and the definition of $E$ each curve in $f(\Gamma_E)$ has one endpoint in $B_I(f(x), Cd)$, where $C$ is absolute, and the other other endpoint in $\mathbb{C} \setminus B_I(f(x), M d)$. Assume $2 \leq C$ and $C^2 < M$. Then by basic modulus estimates and Lemma \ref{confdensmodest} we have
\begin{align*}
\sigma(E)(\log4)^{-1} = \textnormal{Mod}(\Gamma_E) = \textnormal{Mod}(f(\Gamma_E)) \leq C(\log M)^{-1} \leq C\sigma(\Sn)(\log M)^{-1}. 
\end{align*}
If $1 < M \leq C^2$ then trivially
\begin{align*}
\sigma(E) \leq \sigma(\Sn)(\log C^2)(\log M)^{-1}.
\end{align*}

If $|x| \geq 1/4$ then we choose $\Gamma_E$ to be the family of radial segments with one endpoint in $E$ and the other in $B_x \cap \partial B(0,|x|)$. Proceeding like before, the case $1 < M \leq C^2$ is trivial and if $C^2 < M$ then
\begin{align*}
\sigma(E)\log(1/|x|)^{-1} = \textnormal{Mod}(\Gamma_E) = \textnormal{Mod}(f(\Gamma_E)) \leq C (\log M)^{-1}.
\end{align*}
Noting that in this case $\log(1/|x|) \approx \sigma (S_x)$, we are done.
\end{proof}
\begin{lemma}\label{lemma42restatement}
Let $f:\D \rightarrow \mathbb{C}$ be conformal map such that $f(x) \neq 0$ for all $x\in \D$, $\phi$ a growth function and $\delta >0$. There is an absolute constant $C$ such that for each $x \in \D$ and $M > 1,$
\begin{eqnarray*}
\sigma(\{\omega \in S_x : \phi(\delta|f(\omega)|) < \phi(\delta|f(x)|/M)\}) \leq C\sigma(S_x)(\log M)^{-1}.
\end{eqnarray*}
\end{lemma}
\begin{proof}
Let $x \in \D$ and $\omega \in S_x$. Since $\phi$ is strictly increasing, $\phi(\delta|f(\omega)|) < \phi(\delta|f(x)|/M)$ if and only if $|f(\omega)| < |f(x)|/M$. So it suffices to prove the inequality for the case that $\delta = 1$ and $\phi$ is the identity map on $[0, \infty]$. 

Set $E = \{\omega \in S_x :  |f(\omega)| < |f(x)|/M\}$, and choose the curve families $\Gamma_E$ like in the proof of Lemma~\ref{moduluscons}. Note that each curve in $f(\Gamma_E)$ will have one endpoint belonging to $B(f(x), C|f(x)|)$ and the other endpoint in $\mathbb{C} \setminus B(f(x), |f(x)|/M)$, with $C$ an absolute constant. The desired estimate then follows using the same properties of modulus of curve families as in the proof of Lemma~\ref{moduluscons}. 
\end{proof}
We end this section with the following result due to Gehring and Hayman (\cite{GH}). Their theorem says that if ${f: \mathbb{D} \rightarrow f(\D)}$ is conformal then the images of the hyperbolic geodesics in $\mathbb{D}$ are essentially the shortest curves in $f(\D).$ Since the hyperbolic geodesic between $0$ and $x\in\mathbb{D}$ is the radial segment $[0,x]$, we state the following version of their theorem. 
\begin{ghthm}\label{ghthm}
There is a universal constant $C$ with the following property. Suppose $f:\mathbb{D} \rightarrow \mathbb{C}$ is conformal and $\gamma$ is a curve in $\mathbb{D}$ with endpoints 0 and $x \in \mathbb{D}$. Then 
\begin{eqnarray*}
\textnormal{length}(f([0, x])) \leq C\textnormal{length}(f(\gamma)).
\end{eqnarray*} 
\end{ghthm}

\section{$H^\psi_I = H^\psi$ when $\psi$ is doubling}

In this section we prove our main theorem, showing that if $\psi$ is a doubling growth function then $H^\psi$ and $H^\psi_I$ contain the same conformal mappings. The main work is done in Lemma~\ref{mainlemma}, but first we need the following results involving the nontangential maximal functions $f^*$ and $f_I^*$. We handle the classical setting first. 
\begin{lemma}\label{fstarclassical}
Let $\psi$ be a growth function and $f: \D \rightarrow \mathbb{C}$ a conformal mapping. If there exists $\delta > 0$ such that  
\begin{eqnarray*}
\int_{\Sn} \psi (\delta |f(\omega)|) d \sigma < \infty
\end{eqnarray*}
then there is $\epsilon(\delta)> 0$ such that
\begin{eqnarray*}
\int_{\Sn} \psi (\epsilon f^*(\omega)) d \sigma < \infty.
\end{eqnarray*}
\end{lemma}
\begin{proof}
First assume that $f(x) \neq 0$ for all $x\in \D$ and let $\phi = \psi^{1/2}$. By Lemma \ref{lemma42restatement}, there exists an absolute constant $C$ such that
\begin{align*}
\sigma(\{\omega \in S_x : \phi(\delta|f(\omega)|) \geq \phi(\delta|f(x)|/C)\}) \geq \sigma(S_x)/2
\end{align*}
for every $x\in \D$. Then given any $x\in \D$ we have
\begin{eqnarray*}
\int_{S_x} \phi (\delta |f(\omega)|) d\sigma &\geq& \phi(\delta|f(x)|/C) \sigma(\{\omega\in S_x: \phi(\delta |f(\omega)|)\geq \phi(\delta|f(x)|/C)  \})\\ &\geq& \phi(\delta|f(x)|/C) \frac{\sigma(S_x)}2.
\end{eqnarray*}
Let $M$ denote the non-centered Hardy-Littlewood maximal function on $\Sn$. By the previous inequality 
\begin{align*}
\phi(\delta f^*(\omega)/C) \leq 2 M\phi (\delta |f|)(\omega),
\end{align*}
and so recalling that $M$ is a bounded operator on $L^2(\Sn)$ (c.f. \cite{garnett}) we now have
\begin{eqnarray*}
 \int_{\Sn} \psi(\delta f^*(\omega)/C) d\sigma = \int_{\Sn} \phi^2(\delta f^*(\omega)/C) d\sigma \leq 4\int_{\Sn} M^2\phi(\delta |f|)(\omega)d\sigma \\\leq C_1\int_{\Sn} \phi^2(\delta |f(\omega)|) d\sigma = C_1\int_{\Sn} \psi(\delta |f(\omega)|) d\sigma < \infty.
\end{eqnarray*}
This completes the proof for the case that $f(x) \neq 0$ on $\D$. The other case is handled easily by applying the above result to $g(x) = f(x) - y$, where $y$ is some fixed point in $\mathbb{C} \setminus f(\D)$. 
\end{proof}
The intrinsic version of the lemma can be stated in a nicer form than in the classical setting, since there is no need to consider separate cases.
\begin{lemma}\label{ghapp}
There exists a universal constant $C$ such that if $\psi$ is a growth function, $f$ is a conformal mapping of $\D$ and $\delta > 0$ then
\begin{eqnarray*}
\int_{\Sn}\psi(\frac{\delta}{C}f_I^*(\omega))d\sigma \leq \int_{\Sn}\psi(\delta|f(\omega)|_I)d\sigma, 
\end{eqnarray*}
\end{lemma}
\begin{proof}

The Gehring-Hayman theorem and (\ref{harnack}) imply that there is a universal constant $C$ such that
\begin{eqnarray*}
|f(x)|_I \leq |f(t\omega)|_I + d_I(f(t\omega), f(x)) \leq C\textnormal{length}(f([0, \omega))) \leq C|f(\omega)|_I
\end{eqnarray*}
for every $\omega \in \Sn$ and all $x\in \Gamma(\omega)$. The desired inequality follows. 
\end{proof}
Given a growth function $\psi$, if there exists a constant $C$ such that $\psi (2t) \leq C\psi(t)$ for all $t\in[0, \infty]$ then $\psi$ is called \textit{doubling}. We refer to the infimum of all such $C$ as the doubling constant of $\psi$.
\begin{lemma}\label{mainlemma}
Let $\psi$ be a doubling growth function and $f$ a conformal mapping of $\D$. If there is a map $v:\Sn \rightarrow [0, \infty]$ such that $\psi(v(\omega)) \in L^1(\Sn)$  and
\begin{eqnarray*}
\sup_{x\in \Gamma(\omega)} d(f(x), \partial f(\D)) \leq C v(\omega)
\end{eqnarray*}
for some constant $C$ and almost every $\omega \in \Sn$ then $f \in H_I^\psi$.
\end{lemma}
\begin{proof}
Set $U(\lambda) = \{\omega \in \Sn : f_I^*(\omega) > \lambda \}$. We can use the generalized form of the Whitney decomposition \cite[Theorem III.1.3]{whitneyref} to write $U(\lambda)$ as the union of caps
\begin{eqnarray*}
U(\lambda) = \cup S_{x_j},
\end{eqnarray*}
where the caps have uniformly bounded overlap and there is an absolute constant $C$ such that for all $j$
\begin{align}\label{whitneyineq}
\frac{(1 - |x_j|)}{C} \leq d(S_{x_j}, \partial U(\lambda)) \leq C(1 - |x_j|).
\end{align} 

Suppose $\omega \in S_{x_j}$ and $v(\omega) \leq \gamma$. Then the initial assumption along with inequalities (\ref{internaldiamupperbound}) and (\ref{whitneyineq}) imply that there is $\omega' \in \Sn \setminus U(\lambda)$ and a corresponding $x_j' \in \Gamma(\omega')$ such that
\begin{eqnarray}\label{goodlambda}
|f(x_j)|_I \leq d_I(f(x_j), f(x_j')) + |f(x_j')|_I \leq C\gamma + \lambda.
\end{eqnarray}

Let $M > 1$, $\gamma = \frac{\lambda}{(M+1)C}$ and suppose $\omega \in S_{x_j}$ with $v(\omega) \leq \gamma$ and $|f(\omega)|_I > 2\lambda$. Then (\ref{goodlambda}) gives
\begin{eqnarray*}
d_I(f(\omega), f(x_j)) \geq |f(\omega)|_I - |f(x_j)|_I > MC\gamma \geq M d(f(x_j), \partial f(\D)),
\end{eqnarray*}
and so by Lemma \ref{moduluscons}
\begin{gather*}
\sigma(\{\omega \in S_{x_j} : |f(\omega)|_I > 2\lambda\;\; \text{and}\;\; v(\omega) \leq \gamma \})
\\ \leq \sigma(\{\omega \in S_{x_j}: d_I(f(\omega), f(x_j)) > M d(f(x_j), \partial f(\D)) \})
\\ \leq C\sigma(S_{x_j})(\log M)^{-1},
\end{gather*}
where $C$ is an absolute constant. 

If $|f(\omega)|_I > 2\lambda$ then also $f_I^*(\omega) > \lambda$, and so we can apply the above to conclude that
\begin{gather*}
\sigma(\{\omega \in \Sn :  |f(\omega)|_I > 2\lambda\})
\\ \leq \sigma (\{\omega \in U(\lambda) :   |f(\omega)|_I > 2\lambda \;\; \text{and}\;\; v(\omega) \leq \gamma \}) + \sigma(\{\omega \in \Sn : v(\omega) > \gamma \})
\\ \leq C\sum_j \sigma (S_{x_j}) (\log M)^{-1} + \sigma (\{\omega \in \Sn: v(\omega) > \gamma \}) \\ \leq C\sigma (U(\lambda))(\log M)^{-1} + \sigma (\{\omega \in \Sn: v(\omega) > \gamma \}).
\end{gather*}

Thus
\begin{eqnarray*}
\int_{\Sn} \psi(\frac{1}{2}|f(\omega)|_I) d\sigma = \int_0^\infty \psi'(\lambda) \sigma(\{\omega \in \Sn: |f(\omega)|_I > 2\lambda \})d\lambda
\\ \leq \int_0^\infty \psi'(\lambda) (C\sigma (U(\lambda))(\log M)^{-1} + \sigma (\{\omega \in \Sn: v(\omega) > \frac{\lambda}{(M+1)C} \}))d\lambda \\ = C(\log M)^{-1}\int_{\Sn}\psi(f_I^*(\omega))d\sigma + \int_{\Sn} \psi((M+1)C v(\omega)) d\sigma \\ \leq C(\log M)^{-1}\int_{\Sn}\psi(f_I^*(\omega))d\sigma  + C(M, C_\psi)\int_{\Sn} \psi(v(\omega))d\sigma,
\end{eqnarray*}
where $C_\psi$ denotes the doubling constant of $\psi$.

To finish the proof note that we can apply the above to the functions $f_t(x) = f(tx)$ for each $0 < t < 1$: 
\begin{eqnarray*}
\int_{\Sn} \psi(\frac{1}{2}|f_t(\omega)|_I) d\sigma \leq \frac{C}{\log M}\int_{\Sn}\psi(f_{t_I}^*(\omega))d\sigma + C(M, C_\psi)\int_{\Sn} \psi(v(\omega))d\sigma.
\end{eqnarray*}
Then by Lemma \ref{ghapp} there is a new absolute constant $C = C(C_{\psi})$ such that 
\begin{eqnarray*}
\int_{\Sn} \psi(\frac{1}{2}|f_t(\omega)|_I) d\sigma \leq \frac{C}{\log M}\int_{\Sn}\psi(\frac{1}{2}|f_t(\omega)|_I)d\sigma  + C(M, C_\psi)\int_{\Sn} \psi(v(\omega))d\sigma.
\end{eqnarray*}
By choosing $M$ sufficiently large and combining terms we obtain that
\begin{eqnarray*}
\int_{\Sn} \psi(|f(t\omega)|_I) d\sigma \leq C\int_{\Sn} \psi(v(\omega))d\sigma
\end{eqnarray*}
for each $0 < t < 1$, which completes the proof. 
\end{proof}

The proof of Theorem \ref{mainthm} now follows easily. 

\begin{proof}[Proof of Theorem \ref{mainthm}]
If $f\in H^\psi$ then $\psi(f(\omega)), \psi(f^*(\omega))  \in L^1(\Sn)$ by Fatou's Lemma and Lemma \ref{fstarclassical}. If $0 \in \mathbb{C} \setminus f(\D)$ then 
\begin{align*}
\sup_{x\in \Gamma(\omega)}d(f(x), \partial f(\D) \leq f^*(\omega)
\end{align*}
for every $\omega\in \Sn$, and thus $f\in H^\psi_I$ by Lemma \ref{mainlemma}. The case when $0 \in f(\D)$ follows by first applying the result to $g(x) = f(x) - y$, where $y$ is some fixed point in $\mathbb{C} \setminus f(\D)$. 
\end{proof}
\begin{proof}[Proof of Corollary \ref{mainthmcor}]
If $f\in H^\psi$ then $f\in H^\psi_I$ by Theorem \ref{mainthm}, and then by Fatou's Lemma
\begin{eqnarray}\label{intrinsicLpsi}
\int_{\Sn}\psi(|f(\omega)|_I)d\sigma < \infty.
\end{eqnarray}
The Gehring-Hayman Theorem now implies that 
\begin{eqnarray}\label{fprimeint}
\int_{\Sn} \psi \left(\int_0^1 |f'(r\omega)| dr \right) d\sigma < \infty.
\end{eqnarray}
Conversely if (\ref{fprimeint}) holds, then clearly (\ref{intrinsicLpsi}) holds by definition. It follows by Lemma \ref{fstarclassical} that $f\in H^\psi$. 
\end{proof}

\section{Counterexample for non-doubling $\psi$}
In this section we sketch an example to show that the statement of Theorem \ref{mainthm} does not necessarily hold for non-doubling growth functions. The general idea in our example is to use the (quasi-)hyperbolic metric and its (quasi-)invariance to calculate the difference in the growth orders between the maximum modulus
\begin{align*}
M(r,f) = \sup_{|x| = r} |f(x)|
\end{align*}
and the intrinsic maximum modulus
\begin{align*}
M_I(r,f) = \sup_{|x| = r} |f(x)|_I,
\end{align*}
where $f$ maps onto a domain that spirals in a way to maximize $M_I(r,f)$. Then, using the maximum moduli we can show that $f\in H^\psi$ but $f \notin H^\psi_I$ for an appropriately chosen growth function $\psi$.

Recall that the quasi-hyperbolic metric on a domain $\Omega$ is defined as
\begin{eqnarray*}
k(u, v) = \inf \int_\gamma \frac{1}{d(z, \partial \Omega)} |dz|,\;\; u, v\in \Omega,
\end{eqnarray*}
where the infimum is taken over all curves $\gamma \subset \Omega$ with endpoints $u, v$. This metric is quasi-invariant with the hyperbolic metric, and so whenever $f$ is a conformal mapping of $\D$ onto $\Omega$ we have that
\begin{eqnarray}\label{hypmetricinv}
\log\left(\frac{1}{1- |f^{-1}(u)|}\right) \approx k(0, u).
\end{eqnarray}

Let $\alpha \geq 0$ be fixed and $g(t) = t^{\alpha+1}e^{2\pi i t}, t \geq 0,$ be the planar spiral curve whose distance between its $j-1$ and $j$ loops is comparable to $j^\alpha$. Let $\Omega$ be the simply connected domain having $g(t)$ as its boundary, and let $f$ map $\D$ conformally onto $\Omega$, with $f(0)$ in the center of the spiral. Label as $c_j, j \in \mathbb{N},$ the center points of each complete jth loop in $\Omega$, and choose $z_j\in \D$ so that $f(z_j) = c_j$. See Figure 1.
\begin{figure}

\begin{tikzpicture}
\begin{axis}
[axis x line*=center,
    axis y line*=center,xtick=\empty, ytick=\empty, ymin=-25, xmin=-25, ymax=25, xmax=30]
\addplot+[data cs=polarrad,domain=0:5.1,samples=1500,no markers] (2*pi*\x,\x*\x);
\addplot+[mark options={color=black,},mark=*,mark size=1.25pt,data cs=polarrad] coordinates {(0,20.5)};
\addplot+[mark options={color=black,},mark=*,mark size=1.25pt,data cs=polarrad] coordinates {(0,-1.2)};
\node[above] at (axis cs:20.5,0) {\tiny$f(z_j)$};
\node[above left] at (axis cs:-.9,0) {\tiny$f(0)$};
\end{axis}
\end{tikzpicture}
 \caption{}
\end{figure}

It is easy to calculate that
\begin{align*}
|f(z_j)| \approx j^{\alpha +1},\;\;\; |f(z_j)|_I \approx j^{\alpha +2}, \;\;\; \text{and} \;\;\; k(0, f(z_j)) \approx j^2.
\end{align*}
Then by (\ref{hypmetricinv}) 
\begin{eqnarray*}
|f(z_j)| \approx \left(\log\frac{1}{1-|z_j|}\right)^{\frac{1+\alpha}{2}}
\end{eqnarray*}
and
\begin{eqnarray*}
|f(z_j)|_I \approx |f(z_j)|\left(\log\frac{1}{1-|z_j|}\right)^\frac12.
\end{eqnarray*}
It easily follows that 
\begin{eqnarray*}
M(r,f) \approx \left(\log \frac{1}{1-r} \right)^{\frac{1+\alpha}{2}}\;\;\;\textnormal{and}\;\;\;M_I(r,f) \approx \left(\log\frac{1}{1-r} \right)^{\frac{2+\alpha}{2}}.
\end{eqnarray*}
If $\psi(t) = \exp(t^{\frac{2}{1+\alpha}})-1$ then $\psi$ is a growth function that is not doubling. By a theorem in \cite{hpsiqc}, $f\in H^\psi$ if and only if there exists $\delta > 0$ such that $\int_0^1\psi(\delta M(r,f))dr < \infty$. This integral converges in our example by choosing a sufficiently small $\delta$, and so $f\in H^\psi$. 

For the intrinsic case (\ref{internaldiamupperbound}) implies
\begin{eqnarray*}
\int_{\Sn} \psi(\delta |f(r\omega)|_I) d\sigma \gtrsim (1-r)\psi(\delta M_I(r,f)).
\end{eqnarray*}
Thus by the previous calculations
\begin{eqnarray*}
\int_{\Sn} \psi(\delta |f(r\omega)|_I) d\sigma \gtrsim (1-r)\exp\left(C\delta\left(\log \frac{1}{1-r}\right)^{\frac{2+\alpha}{1+\alpha}}\right),
\end{eqnarray*} 
which diverges as $r \rightarrow 1$ no matter the chosen $\delta$, and so $f \notin H^\psi_I$. 
\bibliographystyle{plain}
\bibliography{firstpaperref}
\vspace{1 pc}
\noindent  Pekka Koskela \& Sita Benedict\\
Department of Mathematics and Statistics\\
University of Jyv\"{a}skyl\"{a}\\
P.O. Box 35 (MaD)\\
FI-40014\\
Finland\\

\noindent{\it E-mail addresses}: \texttt{pkoskela@maths.jyu.fi}, \texttt{sita.c.benedict@jyu.fi}
\end{document}